\documentclass[a4paper,11pt]{amsart}
\usepackage[foot]{amsaddr}
\usepackage{geometry}
\usepackage{amsfonts}
\usepackage{amsthm}
\usepackage{amssymb}
\usepackage{amsmath}
\usepackage{comment}
\usepackage{theoremref}
\usepackage{enumerate}
\usepackage{bbm}
\usepackage{bm}
\usepackage{comment}
\usepackage{mathtools}
\usepackage{a4wide}
\usepackage{xcolor}
\usepackage{lipsum}

\makeatletter
\newcommand{\proofpart}[2]{%
    \par
  \addvspace{\medskipamount}%
  \noindent\emph{Step #1: #2}\par\nobreak
  \addvspace{\smallskipamount}%
  \@afterheading
}
\makeatother

\DeclarePairedDelimiter\abs{\lvert}{\rvert}%
\DeclarePairedDelimiter\norm{\lVert}{\rVert}%

\makeatletter
\let\oldabs\abs
\def\abs{\@ifstar{\oldabs}{\oldabs*}}
\let\oldnorm\norm
\def\norm{\@ifstar{\oldnorm}{\oldnorm*}}
\makeatother

\makeatletter
\g@addto@macro\bfseries{\boldmath}
\makeatother

\newcommand{\C}{\mathbb{C}}

\newcommand{\T}{\mathbb{T}}

\newcommand{\Ka}{\mathcal{K}}
\newcommand{\conj}[1]{\overline{#1}}
\newcommand{\D}{\mathbb{D}}
\newcommand{\B}{\mathcal{B}}

\newcommand{\cD}{\conj{\mathbb{D}}}

\newcommand{\m}{\textit{m}}
\newcommand{\hil}{\mathcal{H}}

\newcommand{\supp}[1]{\text{supp}({#1})}

\newtheorem{thm}{Theorem}[section]
\newtheorem{lemma}[thm]{Lemma}

\newtheorem{cor}[thm]{Corollary}
\newtheorem{prop}[thm]{Proposition}

\theoremstyle{definition}

\theoremstyle{definition}

\newcommand{\Addresses}{{
		\bigskip
		\footnotesize
		
		Adem Limani, \\ \textsc{Centre for Mathematical Sciences \\ Lund University \\
		Lund, Sweden}\\
		\texttt{adem.limani@math.lu.se}
		
	}}

\begin{document}
\title{\textbf{Fourier coefficients of normalized Cauchy transforms }} 

\author{Adem Limani} 
\address{Centre for Mathematical Sciences, Lund University, Sweden}
\email{adem.limani@math.lu.se}

\date{\today}

\begin{abstract}
We consider a uniqueness problem concerning the Fourier coefficients of normalized Cauchy transforms. These problems inherently involve proving a simultaneous approximation phenomenon and establishing the existence of cyclic inner functions in certain sequence spaces. Our results have several applications in different directions. First, we offer a new non-probabilistic proof of a classic theorem by Kahane and Katzenelson on simultaneous approximation. Secondly, we demonstrate the absence of uniform admissible majorants of Fourier coefficients in de Branges-Rovnyak spaces.



\end{abstract}

\maketitle
\section{Introduction}\label{SEC:INTRO}

\subsection{Normalized Cauchy transforms} Let $\mu$ be a positive finite Borel measure on the unit-circle $\T$. We denote by $\Ka_{\mu}$ the normalized Cauchy transform acting on an element $f\in L^1(\mu)$ via the formula
\[
\Ka_{\mu}f (z) = \frac{\Ka(fd\mu)(z)}{\Ka(d\mu)(z)}, 
\qquad z\in \D
\]
where $\Ka(fd\mu)(z) = \int_{\T} \frac{f(\zeta)d\mu(\zeta)}{1-\conj{\zeta}z}$ denotes the Cauchy integral of $fd\mu$ and $\D$ denotes the unit-disc in the complex plane. 
Any positive finite Borel measure $\mu$ on $\T$ gives rise to holomorphic self-map $b: \D \to \D$ via the formula 
\[
\frac{1+b(z)}{1-b(z)} = \int_{\T} \frac{\zeta +z}{\zeta-z} d\mu(\zeta) + i\alpha, \qquad z\in \D
\]
where $\alpha$ is a number uniquely determined by the value $b(0)$. A straightforward calculation shows that 
\begin{equation}\label{EQ:KerHb}
\Ka_{\mu} \kappa_{\lambda} (z) = \frac{1-\conj{b(\lambda)}b(z)}{1-\conj{\lambda}z}, \qquad z, \lambda \in \D
\end{equation}
where $\kappa_{\lambda}(\zeta)= (1-\conj{\lambda}\zeta)^{-1}$ denotes the so-called Cauchy kernels. With this formula at hand, the classical Aleksandrov-Clark theory asserts that the associated normalized Cauchy transform $\Ka_{\mu}$ can be extended to a unitary operator from $L^2(\mu)$ to the image $\Ka_{\mu} L^2(\mu)$, which becomes a reproducing kernel Hilbert space of holomorphic functions in $\D$, with kernels given by \eqref{EQ:KerHb}, see \cite{cima2000backward}. These spaces are referred to as de Branges-Rovnyak, typically denoted by $\hil(b)$, and appear as functional models for a certain class of Hilbert spaces contractions. For instance, if $\mu$ is singular wrt $dm$, then the corresponding symbol $\Theta$ for the associated de Branges-space $\hil(\Theta)$ is an inner function: $\lim_{r\to 1-} \abs{\Theta(r\zeta)} =1$ for $dm$-a.e $\zeta \in \T$. These spaces are known as model spaces and are given a distinct designation, as they can be identified as the closed subspace $K_{\Theta} := H^2 \ominus \Theta H^2$, where $H^2$ denotes the Hardy space with the associated norm
\[
\norm{f}^2_{H^2} := \sum_{n=0}^\infty \abs{\widehat{f}(n)}^2 < \infty.
\]
For a detailed treatment of the Hardy spaces and model spaces, we refer the reader to \cite{cima2000backward}. Analogous to the classical Fourier transform, the normalized Cauchy transform $\Ka_{\mu}$ establishes a unitary equivalence between the space $L^2(\mu)$ and $\hil(b)$, with further applications in the spectral theory of Schr\"odinger operators, see \cite{makarov2005meromorphic}. Recently, an intimate link between de Branges-Rovnyak spaces and the spectral theory of subnormal operator was established in \cite{limani2023problem}. In \cite{poltoratski2003maximal}, A. Poltoratski investigated continuity properties of the normalized Cauchy transform and their maximal counterpart, proving that $\Ka_{\mu}$ maps $L^p(d\mu)$ continuously into the Hardy space $H^p$ for $1<p \leq 2$. Here we shall study the image of normalized Cauchy transforms from a different perspective. Our main question of concern is, how "good/bad" can the Fourier coefficients of functions in the images $\Ka_{\mu} L^1(\mu)$ be? The following result stems from the authors recent work with B. Malman, and may be viewed as an extension of a classical theorem in \cite{aleksandrovinv} (see also \cite{dbrcont}) due to A.B. Aleksandrov.

\begin{thm}[See \cite{limani2022abstract}] \thlabel{THM:AlekUa} For any positive finite Borel measure $\mu$ on $\T$, the image $\Ka_{\mu}L^1(\mu)$ contains functions $f$ with uniformly convergent Fourier series on $\T$. In fact, the set of such functions form a dense subspace in the corresponding de Branges-Rovnyak space.
\end{thm}
This means that for any positive finite Borel measure $\mu$, the image $\Ka_{\mu}L^1(d \mu)$ always contains an abundance of functions with fairly nice boundary behavior. We shall investigate whether the Fourier coefficients of $\Ka_{\mu}L^1(d \mu)$ can be better than merely $\ell^2$-summable.

\section{Main results}\label{SEC:MainResults}
\subsection{Fourier coefficients of normalized Cauchy transforms}
Throughout, we shall let $\Phi$ be a so-called Young function, that is, a convex function on $[0,\infty)$ with the properties that $\Phi(0)=0$ and $\Phi(2x) \asymp \Phi(x)$ as $x\to 0$. We define by $\ell^{\Phi}_a$ the associated Orlicz space of holomorphic functions $f$ in $\D$ satisfying
\[
\sum_{n=0}^\infty \Phi \left(\abs{\widehat{f}(n)} \right) < \infty,
\]
where $(\widehat{f}(n))_{n=0}^{\infty}$ denotes the Taylor (Fourier) coefficients of $f$. Equipped with the norm  
\[
\norm{f}_{\Phi} := \inf \left\{M>0: \sum_{n=0}^\infty \Phi \left(\abs{\widehat{f}(n)}/M \right) \leq 1 \right\},
\]
one can show that $\ell^{\Phi}_a$ becomes a separable Banach space, containing the holomorphic polynomials as a dense subset. For further details on these matters, we refer the reader to Lindenstrauss and Tzafriri in \cite{lindenstrauss2013classical}. For instance, when $\Phi(t)=t^p$, then we retrieve the familiar spaces $\ell^p_a$ of holomorphic functions $f$ in $\D$ with 
\[
\norm{f}_p := \left( \sum_{n=0}^\infty \abs{\widehat{f}(n)}^p \right)^{1/p} <\infty. 
\]
With this notation, we have $H^2 = \ell^2_a$. Recall that the classical Clark Theory ensures that image of normalized Cauchy transform under $L^2(\mu)$ always gives rise to functions with $\ell^2$-summable Fourier coefficients, effectively due to the fact that the de Branges-spaces $\hil(b)$ are continuously contained in $H^2$. Our main result confirms that this feature cannot be furnished any further.
\begin{thm}\thlabel{THM:UniqNCT} Let $\Phi$ be a Young function with $\lim_{x\to 0+}\Phi(x)/x^2 = \infty$. Then the following statements hold:
\begin{enumerate}
    \item[(i)] There exists a positive measure $\mu$ singular wrt $dm$ such that
    \[
    \Ka_{\mu} L^1(\mu) \cap \ell^{\Phi}_a =\{0\}.
    \]
    \item[(ii)] There exists a positive absolutely continuous measure $\mu$ such that 
    \[
    \Ka_{\mu} L^1(\mu) \cap \ell^{\Phi}_a =\{0\}.
    \]
\end{enumerate}
    
\end{thm}
The above result is new even for the sequence spaces $\ell^p_a$ with $1\leq p < 2$, hence there exists a positive Borel measures $\mu$ on $\T$, whose normalized Cauchy transform contains no non-trivial element with $\ell^p$-summable Fourier coefficients. Using Clark Theory, we obtain the following immediate corollary, phrased in the framework of de Branges-Rovnyak spaces.

\begin{cor}\thlabel{cor:BM1type} Let $\Phi$ be a Young function with $\Phi(x)/x^2 \to \infty$ as $x \to 0+$. Then the following statements hold: 
\begin{enumerate}
    \item[(i)] There exists a singular inner function $\Theta$, such that for any $f \in K_{\Theta}$, we have 
\[
\sum_{n=0}^\infty \Phi \left(  \abs{\widehat{f}(n)} \right) < \infty \implies f\equiv 0.
\]
    \item[(ii)] There exists an outer function $b$, such that for any $f \in \hil(b)$, we have 
\[
\sum_{n=0}^\infty \Phi \left(  \abs{\widehat{f}(n)} \right) < \infty \implies f\equiv 0.
\]
\end{enumerate}

\end{cor}
We remark that the holomorphic self-map $b$ in the statement $(ii)$ must necessarily be an extreme point in the unit-sphere of the Banach space $H^\infty$ of bounded holomorphic functions in $\D$. This geometric property is characterized by the size condition 
\[
\int_{\T} \log(1-|b|) dm = -\infty,
\]
and typically provides a rough dichotomy in the theory of de Branges-Rovnyak spaces, see \cite{sarasonbook}. Note that \thref{cor:BM1type} answers a question raised by B. Malman in \cite{malman2022cyclic} in the context of $\ell^p_a$-spaces.


\subsection{Cyclic inner functions} \label{SEC:Cyc}
As indicated, the problem of images of normalized Cauchy transform associated to singular measures, boils down to finding a model space which contains no non-trivial function whose Fourier coefficients are $\ell^{\Phi}_a$-summable. A dual reformulation of that problems entails in proving the following result on existence of cyclic inner functions.

\begin{thm} \thlabel{THM:CycinnerlPsi}
Let $\Psi$ be a Young function with $\Psi(x)/x^2 \downarrow 0$ as $x\downarrow 0$. Then there exists a singular inner function $\Theta$ such that the set
\[
\{z^n \Theta(z): n=0,1,2,\dots \}
\]
has a dense linear span in $\ell^{\Psi}_a$.
    
\end{thm}
Again, this result is new even in the classical context of $\ell^p_a$ with $p>2$. \thref{THM:CycinnerlPsi} will be deduced from a different result, which is perhaps interesting in its own right. To this end, we declare a function $w$ on the unit-interval $[0,1)$ to be a majorant if there exists a number $0<\gamma<1$ such that $w(t)/t^{\gamma}$ is non-decreasing on $(0,1]$. Given a majorant $w$, we denote by $\B_0(w)$ the space of holomorphic functions $f$ in $\D$ satisfying 
\[
\lim_{\abs{z}\to 1-} \frac{1-|z|}{w(1-|z|)} f'(z) =0.
\]
The space $\B_0(w)$ may be regarded as the closure of holomorphic polynomials in the Banach space 
\[
\norm{f}_{w} := \abs{f(0) } + \sup_{z \in \D} \frac{1-|z|}{w(1-|z|)} \abs{f'(z)}.
\]
The existence of a cyclic singular inner function in $\B_0$ ($w\equiv 1$) was initially proved in \cite{anderson1991inner}, and further studies of shift invariant subspaces generated by inner functions in $\B_0$ were recently carried out in \cite{limani2023m_z}. The following theorem extends the main theorem in \cite{anderson1991inner}.

\begin{thm}\thlabel{THM:CycinnerBw} Let $w$ be a majorant which violates the square Dini-condition: 
\[
\int_{0}^1 \frac{w^2(t)}{t} dt = + \infty.
\]
Then there exists a singular inner function $\Theta$, such that 
\[
\{z^n \Theta(z): n=0,1,2,\dots \}
\]
has a dense linear span in $\B_0(w)$.
\end{thm}
The result is sharp in the sense that \thref{THM:CycinnerBw} fails drastically for majorants $w$ that satisfy the square Dini-condition. In fact, $\B_0(w)$ does not even contain singular inner function, see Lemma 1 in \cite{dyakonov2002weighted}.






\subsection{Simultaneous approximation and removable sets}\label{SEC:SA}
We now consider a one-sided problem on uniqueness sets for Cauchy integrals in $\ell^{\Phi}_a$. Given a compact subset $K \subset \T$, we denote by $M(K)$ the Banach space of complex finite Borel measures on $\T$, which are supported on $K$. Adapting the notions introduced by Khrushchev in \cite{khrushchev1978problem}, we say that a compact set $K \subset \T$ is \emph{removable for the Cauchy transform} in $\ell^{\Phi}_a$, if for any $\nu \in M(K)$, we have
\[
\sum_{n=0}^\infty \Phi \left( \, \abs{\widehat{\nu}(n)} \, \right) < \infty \implies \nu \equiv 0.
\]
 In other words, $K\subset \T$ is removable for the Cauchy transform in $\ell^{\Phi}_a$ if there exists no non-trivial complex finite Borel measure $\nu$ on $K$, whose Cauchy transform $\Ka (d\mu)$ belongs to $\ell^{\Phi}_a$. Removable sets for Cauchy transforms in function spaces determined by various smoothness conditions of H\"older-type and such, were thoroughly investigated by S. Khrushchev in \cite{khrushchev1978problem}. More recently, removable sets for Cauchy transforms in weighted sequence spaces where considered by B. Malman, see \cite{malman2022cyclic}. It is well-known that the problem of removable sets for Cauchy transforms can be rephrased by means of duality, as a problem of simultaneous approximation in the dual space (see Theorem 1.2 in \cite{khrushchev1978problem}). Our next result is essentially a holomorphic version of a Theorem by Katznelson and Kahane in \cite{kahane1971comportement}, see also Lemma 2 in \cite{kahane2000series}. There, it is proved using probabilistic arguments, but we shall give a proof based on a purely deterministic construction.

\begin{thm}\thlabel{THM:CAUUniqlp} 
For any $\Psi$ be a Young function with $\Psi(t)/t^2 \downarrow 0$ as $t\downarrow 0$. Then there exists compact sets $K \subset \T$ of Lebesgue measure arbitrary close to full, with the properties that for any $f\in \ell^{\Psi}_a$ and any continuous function $g$ on $K$, there exist polynomials $(Q_n)_n$ such that 
\[
\lim_n \norm{Q_n -f}_{\Psi} + \sup_{z\in K} \abs{Q_n(z)-g(z)}=0.
\]
\end{thm}
%

The above theorem implies that holomorphic functions with Fourier coefficients in Orlicz spaces $\ell_a^{\Psi}$ with $\Psi(t)/t^2 \downarrow 0$ as $t\downarrow 0$ may exhibit very erratic boundary behavior. In particular, the Carleson-Hunt Theorem drastically fails for $\ell^{\Psi}_a$. We will occasionally refer to sets satisfying the hypothesis of \thref{THM:CAUUniqlp} as sets satisfying the SA-property wrt $\ell^{\Psi}_a$. Next, we provide an immediate application of our results to Menshov universality. We denote $N$-th Taylor polynomial of an element $f \in \ell^p_a$ by
\[
T_N(f)(\zeta) = \sum_{n=0}^N \widehat{f}(n)\zeta^n, \qquad \zeta \in \T,
\]
regarded as a linear map from $\ell^p_a$ into the Banach space of continuous functions $C(K)$, where $K\subset \T$ is a compact set. A pair $(f,K)$ is said to induce a Menshov-universality on $\ell^p_a$ if for any measurable function $g$ on $K$ and any $f\in \ell^p_a$, there exists holomorphic polynomials $(Q_n)_n$ such that $Q_n$ converges to $f$ in $\ell^p_a$ and $Q_n$ converges to $g$ pointwise $dm$-a.e on $K$. As an immediate corollary of our \thref{THM:CAUUniqlp} on simultaneous approximation, we obtain the main result of Kahane and Nestoridis as Corollary \cite{kahane2000series}.

\begin{cor}[Kahane, Nestoridis, \cite{kahane2000series}] Let $p>2$. Then there exists nowhere dense compact sets $K \subset \T$ of Lebesgue measure arbitrary close to full, and a generic subset $\mathcal{U}_p \subset \ell^p_a$ in the sense of Baire categories, such that any function $f \in \mathcal{U}_p$ induces a Menshov-universality pair $(f,K)$ for $\ell^p_a$.
    
\end{cor}
The new feature of this result its proof is non-probabilistic an relies on the authors former recent work on simultaneous approximation in $\B_0(w)$, which revolves around a special class of inner functions. See \cite{limani2024asymptotic}.

\subsection{Uniformly admissible majorants}
Here, we collect some consequence of our developments to problems of Beurling-Mallavian type in the unit disc. Recall that the classical Beurling-Mallavian multiplier Theorem (for instance, see \cite{mashreghi2006beurling} and references therein) asserts that if $w$ is a uniformly continuous function on the real line $\mathbb{R}$, then there exists a non-trivial function $f\in L^2(\mathbb{R})$ with Fourier transform supported on an interval, and such that 
\[
\abs{f(x)} \leq w(x), \qquad x \in \mathbb{R}
\]
holds, if and only if the size of $w$ is not too small: 
\[
\int_{\mathbb{R}} \frac{\log w(t)}{1+t^2} dt > -\infty
\]
In complex analytic terms, the theorem asserts that $w$ (with some mild regularity assumptions) is an admissible majorant for all holomorphic Paley-Wiener classes in the upper-half plane if and only if $\log w$ is Poisson-summable. Regarding the Paley-Wiener classes as particular subclass of model spaces, one may consider the problem of describing the class of admissible majorants in broader contexts. This was indeed done in \cite{havin2003admissible, havin2003admissible2} by V. Havin and J. Mashreghi, and by K. Dyakonov in the unit disc \cite{dyakonov2002zero}. In our context, we offer the following result on the lack of a uniform admissible majorant for Fourier coefficients of functions in de Branges-Rovnyak spaces.

\begin{thm}\thlabel{THM:BM1type} Let $w_n \downarrow 0$ be a decreasing sequence of positive real numbers. Then there exists a holomorphic self-map $b$ on $\D$, such that for any $f \in \hil(b)$, we have
\[
\abs{f_n} \leq w_n \qquad n=0,1,2,\dots \implies f\equiv 0.
\]
\end{thm}
 In other words, there cannot be a uniform quantitative threshold on the decay of Fourier coefficients among all de Branges-Rovnyak spaces. This should be contrasted with \thref{THM:AlekUa}, which asserts that all de Branges-Rovnyak spaces enjoy the improved uniform qualitative property of containing a plenitude of functions with uniformly convergent Fourier series in $\T$. Again, the holomorphic self-map $b$ appearing the statement will be an extreme point in the unit-sphere of $H^\infty$, and also chosen to be an outer function. A principal partial step towards proving \thref{THM:BM1type} will be the following problem on uniqueness sets for Cauchy integrals, which may be interesting in its own right.

\begin{thm}\thlabel{THM:SAell1w} Let $(w_n)_n$ be a decreasing sequence of real numbers with $w_n \to 0$. Then there exists compact sets $K\subset \T$ of Lebesgue measure arbitrary close to full, such that whenever $\mu \in M(K)$ with 
\[
\sup \left\{ \frac{\abs{\widehat{\mu}(n)}}{w_n}: n=0,1,2, \dots \right\} < \infty \implies \mu \equiv 0.
\]
    
\end{thm}
Roughly speaking, for any sequence $w_n \downarrow 0$, there exists sets of positive Lebesgue measure which cannot support measures whose positive Fourier coefficients are majorized of $w_n$.

\subsection{Organization and notation}
The paper is organized as follows: In section \ref{SEC3}, we collections some dual reformulations of our problems and clarify the connection between uniqueness sets for normalized Cauchy transforms to problems of cyclicity and simultaneous approximation. In Section \ref{SEC4}, we prove \thref{THM:CycinnerBw} on cyclic inner functions in $\B_0(w)$, while Section \ref{SEC5} is devoted to establishing an embedding of the Bloch spaces into Orlicz spaces. The proofs of our main results have been gathered in \ref{SEC6}, where most of them are derived as corollaries of our developments in the preceding sections. 

For two positive numbers $A, B >0$, we will frequently use the notation $A \lesssim B$ to mean that $A \leq cB$ for some positive constant $c>0$. If both $A\lesssim B$ and $B \lesssim A$ hold, we will write $A\asymp B$. Additionally, absolute constants will generally be denoted by $C$, even though the value of $C$ may vary from line to line.


\section{Cauchy dual reformulations}\label{SEC3}
\subsection{Orlicz spaces and duality}
Given a Young function $\Phi$, we define its convex conjugate (Legendre transform) as 
\[
\Phi^*(x) := \inf_{y>0} \left( xy - \Phi(y) \right), \qquad x>0.
\]
It is well-known that $\Phi^*$ is also a Young function and that $(\Phi^*)^* = \Phi$. The notion of convex conjugate gives rise to the following Young-type inequality
\begin{equation}\label{EQ:PhiLeg}
xy \leq \Phi(x) + \Phi^*(y), \qquad x,y>0,
\end{equation}
Hence one can deduce the H\"older-type inequality for the Orlicz spaces $\ell^{\Phi}_a$:
\[
\lim_{r\to 1-} \abs{\int_{\T}f(r\zeta) \conj{g(r\zeta)} dm(\zeta)} \leq 2 \norm{f}_{\ell^{\Phi}_a} \norm{g}_{\ell^{\Phi^*}_a}
\]
Furthermore, this gives rise to the Cauchy-dual relation $\left(\ell^{\Phi}_a\right)' \cong \ell^{\Phi^*}_a$. We refer the reader to \cite{lindenstrauss2013classical} for a detailed treatment of these matters. The following elementary will be useful for our purposes.
\begin{lemma}\thlabel{LEM:Conconj}
If $\Phi$ is a Young function with the property $\Phi(x)/x^2  \uparrow \infty$ as $x \downarrow 0$, then its convex conjugate satisfies $\Phi^*(x)/x^2 \downarrow 0$ as $x \downarrow 0$.
    
\end{lemma}

\begin{proof} 
Fix $0<\varepsilon<1$ and pick $\delta>0$ such that $\Phi(y)\geq y^2 /\varepsilon$ whenever $0<y<\delta$. It follows from the definition of convex conjugate that
\[
\Phi^{*}(x) \leq xy - \Phi(y) \leq y(x-y/\varepsilon), \qquad 0< y <\delta, \qquad x>0.
\]
Now if we confine $0<x< \delta$, then we may choose $y= x\varepsilon/2$, and thus $\Phi^{*}(x) \leq \varepsilon x^2$.   
\end{proof}

\subsection{The role of de Branges-Rovnyak spaces}
In this subsection, we let $\Phi$ denote a Young function satisfying $\lim_{x\to 0+} \Phi(x)/x^2 = +\infty$. Then \thref{LEM:Conconj} ensures that its convex conjugate satisfies $\lim_{x \to 0+} \Phi^*(x)/x^2 =0$, which at its turn implies the chain of containment 
\[
\ell^{\Phi}_a \subset H^2 \subset \ell^{\Phi^*}_a.
\]
Here, we shall make an intimate connection between simultaneous approximation to the theory of de Branges-Rovnyak spaces. To this end, we shall need the following key observation due to S. Khrushchev, phrased in our specific context.

\begin{thm}[Khrushchev's lemma, Theorem 1.2 in \cite{khrushchev1978problem}] A compact set $K \subset \T$ satisfies the SA-property wrt $\ell^{\Phi^*}_a$ if and only if $K$ is removable for the Cauchy transform in $\ell_a^{\Phi}$.
    
\end{thm}

We now prove our main result in this subsection.

\begin{prop}\thlabel{PROP:dbrX'} A compact set $K\subset \T$ satisfies the SA-property wrt $\ell^{\Phi^{*}}_a$ if and only if for any $b: \D \to \D$ holomorphic with $1/b\in H^\infty$ and $\{\zeta \in \T: |b(\zeta)|<1 \} \subseteq K$, we have 
\[
\hil(b) \cap \ell^{\Phi}_a = \{0\}.
\]

\end{prop}
In the statement of the Proposition, we may without loss of generality assume that $m(K)<1$. This implies that $b$ is an extreme point in the unit-sphere of $H^\infty$, which is characterized by the size condition
\[
\int_{\T} \log (1-|b|^2) dm =- \infty.
\]
For instance, see the book by D. Sarason in \cite{sarasonbook}. Meanwhile, the assumption $1/b \in H^\infty$ ensures that $b$ is an outer function.
\begin{proof}[Proof of \thref{PROP:dbrX'}]
Suppose $K$ satisfies the SA-property wrt $\ell_a^{\Phi^*}$, fix a holomorphic self-map $b: \D \to \D$ with $1/b \in H^\infty$ and $\{\zeta \in \T: |b(\zeta)|<1 \} \subseteq K$. Then $b$ is an outer function and an extreme point of the unit-sphere of $H^\infty$. Hence if $f\in \hil(b) \cap \ell_a^{\Phi}$ is an arbitrary element, then according to Proposition 2 in \cite{dbrcont}, there exists a unique function $f_+ \in L^2((1-|b|^2)dm)$ associated to $f$, which satisfies the equation
\[
T_{\conj{b}}(f)(z) = \int_{K} \frac{f_+(\zeta)}{1-\conj{\zeta}z} (1-|b(\zeta)|^2)dm(\zeta), \qquad z\in \D.
\]
Here $T_{\conj{b}}(f) := \Ka (f\conj{b}dm)$ denotes the Toeplitz operator with symbol $\conj{b}$. Applying $T_{\conj{1/b}}$ on both sides of this equation and using the fact that $T_{\conj{a}}\hil(b) \subseteq \hil(b)$ for any $a \in H^\infty$ (for instance, see \cite{sarasonbook}) we obtain
\[
f(z) = \int_{K} \frac{f_+(\zeta)\conj{b^{-1}(\zeta)}}{1-\conj{\zeta}z} (1-|b(\zeta)|^2) dm(\zeta), \qquad z\in \D.
\]
This means that $f \in \hil(b) \cap \, \ell^{\Phi}_a$ is the Cauchy transform of the finite measure 
\[
d\mu(\zeta)= f_+(\zeta)\conj{b^{-1}(\zeta)} (1-|b(\zeta)|^2) dm(\zeta),
\]
which is supported in $K$. But since $K$ was assumed to satisfy the SA-property wrt $\ell^{\Phi^*}_a$, Khrushchev's lemma implies that $f_+=0$, hence $T_{\conj{b}}(f)=0$. Since $b$ is outer, we conclude that $f \equiv 0$. Conversely, if $K$ does not satisfy the SA-property wrt $\ell^{\Phi^*}_a$, then Khrushchev's lemma implies that there exists a non-trivial complex Borel measure $\mu$ supported in $K$, such that $\Ka(d\mu) \in \ell^{\Phi}_a$. Since $\ell^{\Phi}_a \subset H^2$, we must actually have that $d\mu = h dm$ with $h\in L^2(dm)$, thus $\Ka(hdm)\in \ell^{\Phi}_a$. Define the outer function
\[
b(z) = \exp \left( - \int_{K} \frac{\zeta+z}{\zeta-z} dm(\zeta) \right), \qquad z \in \D
\]
which is a holomorphic self-map $b: \D \to \D$ with $1-|b|^2 = 1_K (1-e^{-2})$, hence an extreme-point in the unit-sphere of $H^\infty$. Consider the function
\[
g(z) := \int_{K} \frac{h(\zeta)}{1-\conj{\zeta}z} (1-|b(\zeta)|^2) dm(\zeta), \qquad z\in \D
\]
and note that the assumption on $h$ ensures that $g \in \ell^{\Phi}_a$. But then $f = T_{\conj{1/b}}g \in H^2$ and satisfies the equation $g= T_{\conj{b}}f$. Appealing to Proposition 2 in \cite{dbrcont} once again, we conclude that $f\in \hil(b)$, thus $g\in \hil(b) \cap \ell^{\Phi}_a$.

\end{proof}

Our next observation provides a dual reformulation of the problem of cyclic inner functions.

\begin{prop} \thlabel{PROP:dualcyc} Let $\Theta$ be an inner function. Then $K_\Theta \cap \ell_a^{\Phi} =\{0\}$ if and only if $\Theta$ is cyclic in $\ell^{\Phi^*}_a$.
\end{prop}
\begin{proof} Suppose $\Theta$ is cyclic in $\ell^{\Phi^*}_a$. Fix an arbitrary element $f\in K_{\Theta} \cap \ell^{\Phi}_a$ and pick polynomials $(Q_n)_n$ such that $\norm{\Theta Q_n -f}_{\Phi^*} \to 0$. Then
\[
\int_{\T}\abs{f}^2dm = \abs{ \int_{\T} \conj{f}(f-\Theta Q_n) dm } \leq 2 \norm{f}_{\Phi} \norm{\Theta Q_n -f}_{\Phi^*} \to 0, \qquad n\to \infty.
\]
Therefore, $K_{\Theta} \cap \ell_a^{\Phi} =\{0\}$. If $\Theta$ is not cyclic in $\ell^{\Phi^*}_a$, then the Hahn-Banach separation Theorem immediately implies that there exists a non-trivial element in $K_{\Theta} \cap \ell^{\Phi}_a$.
\end{proof}

\section{Cyclic inner functions in $\B_0(w)$}\label{SEC4}

\subsection{Sharp Control}
Here we devote our efforts to proving \thref{THM:CycinnerBw}. Our main ingredient will be the following remarkable result by Anderson, Fernandez and Shields in \cite{anderson1991inner}, on the existence of positive singular measures with a sharp simultaneous control on modulus of continuity and modulus of smoothness.

\begin{thm}[Theorem 4, \cite{anderson1991inner}]\thlabel{MOCMOS} Let $\alpha, \beta$ be positive continuous functions on $(0,1]$, such that $\alpha$ is decreasing, $\beta$ is increasing with $\alpha(1)\geq 2$, $\beta(1)\leq 1$, and 
\begin{enumerate}
    \item[(i)] $\alpha(t) \to \infty$, $\beta(t) \to 0$ as $t\to 0+$.
    \item[(ii)] $\int_0^1 \beta^2(t) dt/t = \infty$.
    \item[(iii)] $\beta(4t)\leq 2\beta(t)$, $0<t<1/4$.
\end{enumerate}
    Then there exists a singular Borel probability measure $\mu$ on $\T$ satisfying 
\begin{enumerate}
    \item[(1)] $\mu(I) \leq 8 |I| \alpha(|I|)$ for any arc $I \subseteq \T$. 
    \item[(2)] $\abs{\mu(I)-\mu(I')} \leq 36 \abs{I} \beta(|I|)$ for every adjacent arcs $I,I'$ of same lengths.
\end{enumerate}
\end{thm}

\subsection{A criterion for cyclicity}
Our first proposition provides a sufficient condition for a function to be cyclic in $\B_0(w)$, and should be viewed as an adapted version of Theorem 1 in \cite{anderson1991inner}. 

\begin{prop}\thlabel{PROP:SuffCyc} Let $w$ be a majorant and $u,v$ be a continuous non-decreasing functions with the property that $v(t)/u(t) = o(w(t))$ as $t\to 0+$. Then $f$ is cyclic in $\B_0(w)$ if it satisfies the following conditions:
\begin{enumerate}
    \item[(a.)] There exists a constant $c_1>0$ such that
    \[
\abs{f(z)} \geq c_1 u(1-|z|), \qquad z\in \D.
    \]
    \item[(b.)] There exists a constant $c_2>0$ such that 
    \[
    (1-|z|)\abs{f'(z)} \leq c_2v(1-|z|), \qquad z \in \D
    \]
\end{enumerate}
\end{prop}
\begin{proof} The proof is principally inspired from Theorem 1 in \cite{anderson1991inner}, hence we shall only provide a sketch. We shall primarily only assume that $u(t)/v(t) \leq cw(t)$ for some constant $c>0$ independent of $t$, and prove that $(f-f_r)/f_r$ is uniformly bounded in the semi-norm of $\B_0(w)$. To this end, we split the estimates into two terms
\[
\frac{d}{dz}\left(\frac{f(z)-f(rz)}{f(rz)}\right) = \frac{f'(z)-rf'(rz)}{f(rz)} - r \frac{\left(f(z)-f(rz)\right)f'(rz)}{f(rz)^2} :=A_1(z) + A_2(z).
\]
\proofpart{1}{Estimating $A_1$:}
Note that
\[
\frac{1-|z|}{w(1-|z|)} |A_1(z)| \leq \frac{1-|z|}{w(1-|z|)} \frac{\abs{f'(z)}}{\abs{f(rz)}} +\frac{1-|z|}{w(1-|z|)} \frac{\abs{f'(rz)}}{\abs{f(rz)}} := A + B.
\]
In order to estimate the quantity $A$ we use the fact that $u$ is non-decreasing and the assumption that $v/u\leq cw$ on $(0,1]$, to obtain 
\[
A \leq c_2 /c_1 \frac{v(1-|z|)}{u(1-r|z|)} \frac{1}{w(1-|z|)} \leq cc_2/c_1.
\]
For quantity $B$, we again use the assumption $v/u\leq cw$ on $(0,1]$, but now together with the fact that $w(t)/t$ is non-decreasing on $(0,1]$. This implies
\[
B \leq c_2/c_1 \frac{1}{w(1-|z|)} \frac{1-|z|}{1-r|z|} \frac{v(1-r|z|)}{u(1-r|z|)} \leq cc_2 /c_1 \frac{w(1-r|z|)}{1-r|z|} \frac{1-|z|}{w(1-|z|)} \leq cc_2/c_1.
\]
Therefore, the term $A_1$ has uniformly bounded the $\B(w)$ semi-norm.
\proofpart{2}{Estimating $A_2$:} For the second term, we shall need the following straightforward estimate, which stems from the fundamental theorem of calculus:
\begin{equation}\label{oscest}
\abs{f(z)-f(rz)} \leq c_2 \int_{r|z|}^{|z|} \frac{v(1-t)}{1-t}dt \leq c_2 v(1-r|z|) \log \frac{1-r|z|}{1-|z|}, \qquad z \in \D.
\end{equation}
With this at hand, we proceed by
\begin{multline*}
    \frac{1-|z|}{w(1-|z|)} \abs{A_2(z)} =  \frac{1-|z|}{w(1-|z|)} \abs{f(z)-f(rz)} \frac{\abs{f'(rz)}}{\abs{f(rz)}^2} \leq \\
    ( c_2/c_1)^2\frac{1-|z|}{1-r|z|} \log \frac{1-r|z|}{1-|z|} \frac{v^2(1-r|z|)}{u^2(1-r|z|)} \frac{1}{w(1-|z|)} \leq \\
    c^2 (c_2/c_1)^2 \frac{w(1-r|z|)}{w(1-|z|)} \frac{1-|z|}{1-r|z|} \log \frac{1-r|z|}{1-|z|}.
\end{multline*}
Utilizing that $w(t)/t^{\gamma}$ is non-increasing for some $0<\gamma <1$, we obtain
\[
\frac{w(1-r|z|)}{w(1-|z|)} \frac{1-|z|}{1-r|z|} \leq \left(\frac{1-|z|}{1-r|z|}\right)^{1-\gamma}.
\]
This observation together with the fact that $x^{\gamma-1}\log x$ is uniformly bounded for $x\geq 1$ gives the desired bound for $A_2$. We have thus proved that 
\[
\sup_{0<r<1} \norm{f/f_r -1}_{w} < \infty.
\]
By means of substitution $w$ with a majorant $\widetilde{w}$ satisfying $\widetilde{w}(t)/w(t) \to 0$ as $t\to 0+$, we may repeat the same argument as above to deduce that $f/f_r$ is uniformly bounded in the norm of $\B_{\widetilde{w}}$. Indeed, this is possible to achieve, if we now utilize the assumption that $v(t)/u(t) = o(w(t))$ as $t\to 0+$. Appealing to the compactness of the inclusion map $\B_{\widetilde{w}}$ into $\B_0(w)$, it then follows that there exists a subsequence $\{r_n\}_n$ with $r_n \to 1-$, such that 
\[
\lim_{n\to \infty} \norm{ f/f_{r_n} -1}_{w} = 0.
\]
Since $1/f_r$ extend holomorphically across $\T$, it easy to pass to holomorphic polynomials $(Q_n)_n$ with $Q_n f \to 1$ in $\B_0(w)$. In conjunction with the fact that the polynomials are dense in $\B_0(w)$, we conclude that $f$ is cyclic in $\B_0(w)$.

\end{proof}

We are now ready to prove the main result in this section.

\begin{proof}[Proof of \thref{THM:CycinnerBw}]
Let $\alpha, \beta$ be a pair of functions satisfying the hypothesis of \thref{MOCMOS}, which are to be specified in a moment, and let $\mu$ denote the corresponding positive Borel measure on $\T$ singular wrt $dm$ which stems from \thref{MOCMOS}. We argue that the associated singular inner function $\Theta = S_{\mu}$ is cyclic in $\B_0(w)$. Note that property $(1)$ in \thref{MOCMOS}, in conjunction with a standard estimates of the Poisson kernel, implies that there exists an absolute constant $c_1>0$ such that 
\[
\abs{S_{\mu}(z)} \geq \exp \left( - c_1 \alpha(1-|z|) \right), \qquad z \in \D. 
\]
Meanwhile, condition $(2)$ in \thref{MOCMOS} implies that there exists a constant $c_2>0$ such that
\[
(1-|z|)\abs{S_{\mu}'(z)}\leq 2(1-|z|) \abs{H'(\mu)(z)} \leq c_2 \beta(1-|z|), \qquad z \in \D.
\]
For further details on these matters, see Theorem 2 and Lemma 4 in \cite{anderson1991inner}. According to \thref{PROP:SuffCyc}, it suffices to verify that the functions $\alpha, \beta$ can be chosen in such a way that 
\begin{equation}\label{EQ:betaalpha}
\beta(t) / \exp(-c_1 \alpha(t) ) = o(w(t)), \qquad t\to 0+.
\end{equation}
Note that the hypothesis in \thref{MOCMOS} essentially puts no serious restriction on $\alpha$, hence by means of slightly enlarging the constant $c_1>0$, it suffices to choose $\alpha, \beta$ so that 
\[
\beta(t) / \exp(-c_1 \alpha(t) ) \asymp w(t), \qquad t\to 0+.
\]
Recalling that $\beta^2$ must fail the Dini-condition, we must also require that 
\begin{equation}\label{EQ:w^2alpha}
\int_0^1 w^2(t) \exp(-2c_1 \alpha(t)) \frac{dt}{t} = +\infty.
\end{equation}
Since $w^2$ was assumed to fail the Dini-condition itself, we can pick $\alpha$ to tend to infinity at an arbitrary slow rate, in such a way that \eqref{EQ:w^2alpha} holds. For such a choice of $\alpha$, we may, for instance, take $\beta = w \exp(-c_1 \alpha)$. Here, we note that doubling properties of $\alpha$ and $w$ carry over to doubling properties of $\beta$, hence the additional condition $(iii)$ on $\beta$ can easily be met be means of slightly modifying $\alpha$ (if necessary). We omit the details. 
\end{proof}

\section{Orlicz embeddings of Bloch spaces}\label{SEC5}
The following embedding result is the center of attention in this section, and is interesting in its own right. 

\begin{thm}\thlabel{THM:Bwlp} Let $\Psi$ be a continuous function in $[0,\infty)$ with the property that $\Psi(x)/x^2$ is increasing and tends to $0$ as $x\to 0+$. Then there exists a majorant $w$ with the property that 
\[
\int_0^1 \frac{w^2(t)}{t} dt = \infty,
\]
such that $\B_0(w)$ is continuously contained in $\ell^{\Psi}_a$.
    
\end{thm}

This result hinges on the following elementary lemma, which is the first part of the statement in \thref{THM:Bwlp}. 

\begin{lemma}\thlabel{LEM:A1} Let $\Psi$ be a continuous function on $[0,\infty)$ with the property that $\Psi(x)/x^2 \downarrow 0$ as $x\downarrow 0$. Then there exists a majorant $w$ with the properties that 
\begin{enumerate}
    \item[(i)] $\int_0^1 \frac{w^2(t)}{t}dt = \infty$.
    \item[(ii)] $\sum_{n=0}^\infty \Psi \left( \, w(2^{-n} ) \, \right) < \infty$.
\end{enumerate}
\end{lemma}

\begin{proof} Note that condition $(i)$ on a majorant $w$ is equivalent to the assumption
\[
\sum_{n=0}^\infty w^2(2^{-n} ) \asymp \sum_{n=0}^\infty \int_{2^{-(n+1)}}^{2^{-n}} \frac{w^2(t)}{t}dt = \int_0^1 \frac{w^2(t)}{t}dt =  \infty,
\]
Set $\alpha(t) := \Psi(\sqrt{t})$ and note that it suffices to construct a decreasing sequence of positive numbers $(w_n)_n$ with the following properties:
\begin{enumerate}
    \item[(a.)] $\sum_n w_n = \infty$.
    \item[(b.)] $\sum_n \alpha(w_n) < \infty$.
    \item[(c.)] There exists a universal constant $C>1$, such that $w_n \leq Cw_{n+1}$ for all $n$.
\end{enumerate}
Having done so, define $w$ as the piecewise linear function with graph-nodes $w(2^{-n}) = \sqrt{w_n}$ for $n=0,1, 2, \dots$. Indeed, the conditions $(a.)$ and $(b.)$ ensure that $(i)$ and $(ii)$ hold, respectively. To see that $(c.)$ ensures that $w$ is a majorant, let $0<s<t\leq 1$ and let $m$ be the largest integer with $s\leq (1+C)^{-m}$ while $n$ is the smallest integer with $t \geq (1+C)^{-n}$. Hence $m\geq n$, and pick $0<\gamma <1$ such that $(1+C)^{-\gamma} C \leq 1$. It follows that
\[
\frac{w^2(t)}{t^{\gamma}} \leq (1+C)^{n\gamma} w_m \leq (1+C)^{-\gamma(m-n)} C^{(n-m)}  \frac{w_m}{(1+C)^{-m\gamma}} \leq C_{\gamma} \frac{w^2(s)}{s^\gamma},
\]
for some constant $C_{\gamma}>0$ only depending on $\gamma$. This verifies that $w$ indeed defines a majorant. We shall primarily show that there exists a sequence $(w_n)$ satisfying $(a.)$ and $(b.)$. Since by assumption $\alpha(t)/t$ decreases to $0$ as $t\to 0+$, there exists a (unique) decreasing sequence $(t_n)_n$ with $\alpha(t_n)/t_n = 1/n^2$. For each $n$, consider a uniform partition of $(t_{n+1},t_n]$ of $N_n$ points $t_{n,k} = t_{n+1} + \frac{k}{N_n}(t_n - t_{n+1})$ with $k=1, \dots, N_n$. Note that
\[
\sum_{k=0}^{N_n} t_{n,k} \asymp t_n N_n , \qquad n=1,2,\dots
\]
Choose $N_n \asymp 1/t_n$ and let $(w_n)_n$ be the joint sequence $(t_{n,k})_{k=1}^{N_n}$ $n=1,2, \dots$. It follows that 
\[
\sum_{n} w_n = \sum_{n=1}^\infty \sum_{k=1}^{N_n} t_{n,k} \gtrsim \sum_{n=1}^\infty 1 = +\infty,
\]
while 
\[
\sum_{n} \alpha(w_n) = \sum_{n=1}^\infty \sum_{k=1}^{N_n} \frac{\alpha(t_{n,k})}{t_{n,k}} t_{n,k} \leq \sum_{n=1}^\infty \frac{1}{n^2} \sum_{k=1}^{N_n} t_{n,k} \lesssim \sum_{n=1}^\infty \frac{1}{n^2} < \infty.
\]
This construction yields a decreasing sequence $(w_n)$ which satisfies the conditions $(a.)$ and $(b.)$. If necessary, we may add a sequence of positive real numbers $\widetilde{w}_n \in (2^{-(n+1)},2^{-n}]$ for each $n=1,2,\dots$ to the sequence $(w_n)_n$, and re-label it in decreasing order. This ensures that the quotient between two consecutive numbers does not exceed $2$, hence $(c.)$ holds, while $(a.)$ and $(b.)$ still remain intact. 
\end{proof}

Before turning to the proof of \thref{THM:Bwlp}, we shall need the following estimate on the Fourier coefficients of Bloch functions.

\begin{lemma}\thlabel{LEM:BwFourier} There exists a constant $C>0$, such that for any $f \in \B_0(w)$, the following estimates hold:

\begin{equation}\label{EQ:sumBw1}
\frac{1}{N^2}\sum_{n=1}^N n^2 \abs{\widehat{f}(n)}^2 \leq C\norm{f}^2_{w}  w(1/N)^2, \qquad N=1,2,3,\dots
\end{equation}
In particular, we have 
\[
\abs{\widehat{f}(n)} \leq C w(1/n) \norm{f}_{w}, \qquad n=1,2,3,\dots. 
\]
    
\end{lemma}

\begin{proof} Note that for any holomorphic function $f(z)= \sum_{n=0}^\infty \widehat{f}(n) z^n$ on $\D$, an application of Parseval's theorem gives
\begin{equation}\label{EQ:Parseval}
\sum_{n=1}^{\infty} n^2 \abs{\widehat{f}(n)}^2 r^{2(n-1)} = \int_{\T} \abs{f'(r\zeta)}^2 dm(\zeta)  \qquad 0<r<1,
\end{equation}
Fix a positive integer $N>0$ and choose $r_N = 1-1/N$, which yields
\[
\sum_{n=1}^N n^2 \abs{\widehat{f}(n)}^2 \leq C \int_{\T} \abs{f'(r_N \zeta)}^2 dm(\zeta) \leq C\norm{f}^2_{w} N^2 w(1/N)^2, \qquad N=1,2,3,\dots,
\]
This proves \eqref{EQ:sumBw1} and the second statement follows from omitting all terms in the above sum, except for the $N$-th term. 
\end{proof}



\begin{proof}[Proof of \thref{THM:Bwlp}] 
According to \thref{LEM:A1}, there exists a majorant $w$ with $\sum_n w^2(2^{-n}) = \infty$, but such that $\sum_n \Psi(w(2^{-n})) < \infty$. Set $\alpha(t) := \Psi(t)/t^2$, which assumed to be a continuous increasing function on $[0,1)$ with $\alpha(0)=0$. Fix an arbitrary function with $\norm{f}_{w}=1$. Using the majorant assumption on $w$, and \thref{LEM:BwFourier}, we obtain 
\[
\alpha \left( \abs{\widehat{f}(j)} \right) \leq \alpha \left( cw(1/j) \right) \leq \alpha \left( c'w(2^{-n}) \right) , \qquad 2^{n-1}\leq j < 2^n,
\]
where $c,c'>0$ are numerical constants. This observation in conjunction with another application of \thref{LEM:BwFourier} implies 

\begin{equation*}
\begin{split}
\sum_{n=1}^\infty \Psi \left( \abs{\widehat{f}(n)} \right) = \sum_{n=1}^\infty \sum_{2^{n-1}\leq j < 2^n } \Psi \left( \abs{\widehat{f}(j)} \right)  & \leq  \sum_{n=1}^\infty \alpha \left( c' w(2^{-n}) \right) \left( 2^{-2n} \sum_{2^{n-1}\leq j < 2^n}  \abs{\widehat{f}(j)}^2 j^2 \right) \\ \leq 
C \sum_{n=1}^\infty \Psi \left( c' w(2^{-n}) \right). 
\end{split}
\end{equation*}
Since the right-hand side is finite, the proof is complete.

\end{proof}

\section{Proof of main results}\label{SEC6}

\subsection{Simultaneous approximation and cyclic inner functions} 
We are now ready to deduce \thref{THM:CAUUniqlp} and \thref{THM:CycinnerlPsi}, respectively. 

\begin{proof}[Proof of \thref{THM:CAUUniqlp}] \thref{THM:Bwlp} ensures the existence of a majorant $w$ which violates the square Dini-condition, and gives rise to the continuous containment of $\B_0(w)$ in $\ell^{\Psi}_a$. Hence, there exists a constant $C>0$, possibly depending on $w$ and $\Psi$, such that 
\[
\norm{f}_{\Psi} \leq C \norm{f}_w, \qquad f \in \B_0(w).
\]
According Theorem 1 in \cite{limani2024asymptotic}, there exists compact sets $K \subset \T$ of Lebesgue measure arbitrary close to full, such that for any polynomial $f$ and any continuous function $g$ on $K$, there exists holomorphic polynomials $(Q_n)_n$ such that 
\[
\lim_n \norm{Q_n-f}_w + \sup_{\zeta \in K} \abs{Q_n(\zeta)-g(\zeta)} =0.
\]
Now using Theorem, we get the same conclusion for $\ell^{\Psi}_a$ with $f\in \B_0(w)$. But since the polynomials are dense in $\ell^{\Psi}_a$, the statements remains true for all $f \in \ell^{\Psi}_a$.
\end{proof}

\begin{proof}[Proof of \thref{THM:CycinnerlPsi}]
Note that since the polynomials are dense in $\ell^{\Psi}_a$, it suffices to show that there are polynomials $(Q_n)_n$ with $\Theta Q_j \to 1$ in $\ell^{\Phi}_a$. However, this is an immediate consequence of \thref{THM:Bwlp} in conjunction with \thref{THM:CycinnerBw}.

\end{proof}

\subsection{Images of normalized Cauchy transforms}
Note that if $\Psi$ is a Young function with $\Psi(x)/x^2 \uparrow \infty$ as $x\to 0$, then $\Ka_{\mu}L^1(\mu) \cap \ell_a^{\Psi} = \Ka_{\mu}L^2(\mu) \cap \ell_a^{\Psi}$. Indeed, this follows from the containment $\ell_a^{\Psi} \subset H^2$ in conjunction with $\Ka_{\mu}$ being a injective on $L^2(\mu)$. This implies that \thref{cor:BM1type} is an equivalent reformulation of \thref{THM:UniqNCT}. However, part $(i)$ of \thref{cor:BM1type} is now an immediate consequence of \thref{THM:Bwlp} and \thref{THM:CycinnerBw}, while part $(ii)$ follows from \thref{THM:Bwlp} in conjunction with \thref{THM:CAUUniqlp}.

\subsection{Non-uniform admissible majorants}
It only remains to prove \thref{THM:BM1type} and \thref{THM:SAell1w} on non-admissible majorants. To this end, let $w=(w_n)_n$ be decreasing positive real numbers which tend to zero, and denote by $\ell^1_a(w)$ the Banach space of holomorphic functions $f$ in $\D$, equipped with the norm
\[
\norm{f}_{\ell^1(w)} := \sum_{n=0}^\infty \abs{\widehat{f}(n)}w_n < \infty.
\]
The following lemma will be our principal device in this section.
\begin{lemma} \thlabel{PROP:SAl1w} Let $(w_n)_n$ be positive numbers with $w_n \downarrow 0$. Then for any $0<\delta<1 $, there exists a compact set $K\subset \T$ with $1-\delta <m(K)<1$, and holomorphic polynomials $(Q_n)_n$ satisfying the following properties
\begin{enumerate}
    \item[(i)] $\lim_n \norm{Q_n}_{\ell^1(w)} =0$
    \item[(ii)] $\lim_n \sup_{\zeta \in K} \abs{Q_n(\zeta)-1}=0$.

\end{enumerate}
    
\end{lemma}
\begin{proof}
The following construction is inspired by S. Khrushchev in \cite{khrushchev1978problem}. Fix two numbers $\gamma, \delta \in (0,1)$, and let $A_\delta \subset \T$ be a closed arc of length $1-\delta$. Define a smooth real-valued functions $\psi_{\gamma,\delta}$ on $\T$ satisfying the properties:
\begin{enumerate}
    \item[(i)] $\int_{\T} \psi_{\gamma, \delta}(\zeta) dm(\zeta) = 0$,
    \item[(ii)] $\psi_{\gamma, \delta}(\zeta) = \log \gamma$, \qquad $\zeta \in A_{\delta}$.
\end{enumerate}
we now form the outer function $G_{\gamma, \delta}$ defined by 
\[
G_{\gamma, \delta}(z) = \exp \left( \int_{\T} \frac{\zeta +z}{\zeta -z} \psi_{\gamma, \delta}(\zeta) dm(\zeta) \right) , \qquad z \in \D.
\]
Set $F_{\gamma, \delta}= 1-G_{\gamma,\delta}$ which is easily seen to be bounded and holomorphic in $\D$, which also extends to a smooth function to $\T$. We now record the following properties:
\begin{enumerate}
    \item[(i)] $F_{\gamma, \delta}(0)=0$.
    \item[(ii)] $\abs{F_{\gamma, \delta}(\zeta)-1} = \gamma$, \qquad $\zeta \in A_{\delta}$.
    \item[(iii)] $$N(\gamma, \delta) := \left( \int_{\T} \abs{F_{\gamma,\delta}'(\zeta)}^2 dm(\zeta) \right)^{1/2}$$ tending to $+\infty$ as either $\delta, \gamma \to 0+$.
    
\end{enumerate}
Let $(m_n)_n$ be an increasing sequence of positive integers with $m_n \to \infty$ to be chosen in a moment, and consider the functions
\[
f_{n,\gamma, \delta}(z) := F_{\gamma,\delta}(z^{m_n}), \qquad z\in \D.
\]
Let $E_{n,\delta}$ denote the pre-image of $A_{\delta}$ under the monomial $z^{m_n}$, and note that $m(E_{n,\delta}) = m(A_{\delta})=1-\delta$. It follows from $(ii)$ that 
\begin{equation}\label{EQ:fnE}
\abs{f_{n,\gamma, \delta}(\zeta)-1}= \gamma, \qquad \zeta \in E_{n,\delta}.
\end{equation}
We now give an estimate on the Fourier coefficients of $f_{n,\gamma, \delta}$ and recall that $f_{n,\gamma, \delta}(0)=0$. To this end, note that
\begin{equation}\label{EQ:l1w}
\sum_{k=0}^\infty \abs{\widehat{f}_{n,\gamma, \delta}(k)}w_k = \sum_{k=1}^\infty \abs{\widehat{F}_{\gamma, \delta}(k)}w_{m_n k} \leq w_{m_n} \left( \sum_{k=1}^\infty \frac{1}{k^2} \right)^{1/2} N(\gamma, \delta).
\end{equation}
Here, we used the monotonicity assumption on the sequence $(w_n)_n$ and the Cauchy-Schwarz inequality. Pick a sequence $(\delta_n)_n$ of positive numbers with the property that $\sum_n \delta_n \leq \delta$ and let $(\gamma_n)_n$ be any sequence of positive numbers tending to zero. Setting $K := \cap_n E_{n, \delta_n}$, and note that the assumption on the $\delta_n$'s imply that $m(K) \geq 1-\delta$. We now choose the positive integers $(m_n)_n$ to be large enough so that 
\[
w_{m_n} N(\gamma_n, \delta_n) \to 0, \qquad n\to \infty,
\]
and for the sake of brevity set $f_n := f_{n,\gamma_n, \delta_n}$. It follows from \eqref{EQ:fnE} that $f_n \to 1$ uniformly on $K$, while \eqref{EQ:l1w} implies that $f_n \to 0$ in $\ell^1_a(w)$. Recalling that each $f_n$ is smooth up to $\T$, we can approximate $f_n, f'_n$ uniformly on $\cD$ by holomorphic polynomials. It follows from the Cauchy-Schwartz type estimate in \eqref{EQ:l1w} that such polynomials satisfy the properties $(i)$ and $(ii)$ in the statement, hence the proof is complete.
\end{proof}

 We now turn to our main task.

\begin{proof}[Proof of \thref{THM:SAell1w}] According to \thref{PROP:SAl1w}, we can for any small desirable number $\delta >0$, find a compact set $K \subset \T$ with $1-\delta<m(K)<1$, and corresponding holomorphic functions $(f_n)_n$ that satisfy $f_n \to 1$ uniformly on $K$, while $f_n \to 0$ in $\ell^1_a(w)$. Now let $\mu \in M(K)$ be an arbitrary element with $\sup_{j\geq 0} \frac{\abs{\widehat{\mu}(j)}}{w_j} < \infty$. Then for any integer $m \geq 0$, we have
\begin{equation*}
\begin{split}
\abs{\widehat{\mu}(m)} = \lim_n \abs{\int_K \conj{f_n(\zeta)} \zeta^{-m} d\mu(\zeta)} = \lim_n \abs{\int_{\T} \conj{ f_n(\zeta)\zeta^m} d\mu(\zeta)} = \\ 
\lim_n \abs{\sum_{j=0}^\infty \conj{\widehat{f_n}(j+m)} \widehat{\mu}(j)} \leq \limsup_n \norm{f_n}_{\ell^1(w)} \sup_{j\geq 0} \frac{\abs{\widehat{\mu}(j)}}{w_j}=0.
\end{split}
\end{equation*}
According to the F. and M. Riesz Theorem, there exists an function $g \in H^1$, such that $\conj{\mu}= g dm$. Now since $m(K)<1$, we conclude that $g$ must vanish on a set positive Lebesgue measure, hence $\mu \equiv 0$. This completes the proof.

\end{proof}

We remark that the duality argument above actually shows that for any $f\in \ell^1_a(w)$ and any function $g$ continuous on $K$, there exists holomorphic polynomials $(Q_n)_n$ with $Q_n \to f$ in $\ell^1_a(w)$, while $Q_n \to g$ uniformly on $K$. The proof of \thref{THM:BM1type} is now an immediate consequence of this remark in conjunction with \thref{PROP:dbrX'}.

\bibliographystyle{siam}
\bibliography{mybib}

\Addresses

\end{document}